\documentclass{amsart}
\usepackage{amstext}
\usepackage{amsthm}

\usepackage{amssymb}
\usepackage{latexsym}
\usepackage{amsmath}
\usepackage{amscd}
\usepackage{amsfonts}%
\usepackage{graphicx}

\usepackage{mathrsfs}

\usepackage{tikz}
\usepackage{tikz-cd}
\usetikzlibrary{matrix,arrows}

\newtheorem{theorem}{Theorem}[section]

\newtheorem{lemma}[theorem]{Lemma}

\newtheorem{proposition}[theorem]{Proposition}

\theoremstyle{remark}

\numberwithin{equation}{section}

\newcommand{\Hcal}{\mathscr{H}}

\newcommand{\Lcal}{\mathscr{L}}
\newcommand{\Mcal}{\mathscr{M}}

\newcommand{\Rcal}{\mathscr{R}}

\newcommand{\Z}{\mathbb{Z}}
\newcommand{\C}{\mathbb{C}}

\newcommand{\Q}{\mathbb{Q}}

\newcommand{\A}{\mathbb{A}}

\newcommand{\Nr}{\mathrm{Nr}}
\newcommand{\Val}{\mathrm{Val}}

  \DeclareFontFamily{U}{wncy}{}
    \DeclareFontShape{U}{wncy}{m}{n}{<->wncyr10}{}
    \DeclareSymbolFont{mcy}{U}{wncy}{m}{n}
    \DeclareMathSymbol{\Sha}{\mathord}{mcy}{"58}

\begin{document}
\title[The Diophantine problem for exponential polynomials]{The Diophantine problem for rings of exponential polynomials}

\author[D. Chompitaki]{Dimitra Chompitaki}
\address{Department of Mathematics and Applied Mathematics, University of Crete-Heraklion, Greece}
\email[D. Chompitaki]{d.hobitaki@gmail.com}%

\author[N. Garcia-Fritz]{Natalia Garcia-Fritz}
\address{ Departamento de Matem\'aticas,
Pontificia Universidad Cat\'olica de Chile.
Facultad de Matem\'aticas,
4860 Av.\ Vicu\~na Mackenna,
Macul, RM, Chile}
\email[N. Garcia-Fritz]{natalia.garcia@mat.uc.cl}%

\author[H. Pasten]{Hector Pasten}
\address{ Departamento de Matem\'aticas,
Pontificia Universidad Cat\'olica de Chile.
Facultad de Matem\'aticas,
4860 Av.\ Vicu\~na Mackenna,
Macul, RM, Chile}
\email[H. Pasten]{hpasten@gmail.com}%

\author[T. Pheidas]{Thanases Pheidas}
\address{Department of Mathematics and Applied Mathematics, University of Crete-Heraklion, Greece}
\email[T. Pheidas]{pheidas@uoc.gr}%

\author[X. Vidaux]{Xavier Vidaux}
\address{ Universidad de Concepci\'on, Facultad de Ciencias F\'isicas y Matem\'aticas, Departamento de Matem\'atica. Casilla 160 C}
\email[X. Vidaux]{xvidaux@udec.cl}%

\thanks{D. C. was supported by the Hellenic Foundation for Research and Innovation
(HFRI) and the General Secretariat for Research and Technology (GSRT), under the HFRI PhD
Fellowship grant (GA. no.2321). }
\thanks{N. G.-F. was supported by ANID Fondecyt Regular grant 1211004, ANID (ex CONICYT) Fondecyt Iniciacion en Investigacion 11170192 and ANID (ex CONICYT) PAI grant 79170039.}
\thanks{H. P. was supported by ANID (ex CONICYT) Fondecyt Regular grant 1190442.}
\thanks{D. C. and T. P. were supported in part by a European Union Erasmus+ project through the University of Crete, Greece, and the University of Concepci\'on, Chile.}
\thanks{ X. V. was supported in part by ANID Fondecyt Regular grant 1210329 and ANID (ex CONICYT) Fondecyt Regular grant 1170315.}
\date{\today}

\begin{abstract} One of the main open problems regarding decidability of the existential theory of rings is the analogue of Hilbert's Tenth Problem (HTP) for the ring of entire holomorphic functions in one variable. In the direction of a negative solution, we prove unsolvability of HTP for rings of exponential polynomials. This provides the first known case of HTP for a ring of entire holomorphic functions in one variable strictly containing the polynomials. The technique of proof consists of an interaction between Arithmetic, Analysis, Logic, and Functional Transcendence. 
\end{abstract}

\maketitle



\section{Introduction}

By the work of Davis-Putnam-Robinson \cite{DPR} and Matijasevich \cite{Matijasevich1970} we know that Hilbert's tenth problem is unsolvable. More generally, let $A$ be a ring and let $A_0$ be a recursive subring of $A$. The analogue of Hilbert's tenth problem for $A$ with coefficients in $A_0$ asks for an algorithm which decides the solvability in $A$ of polynomial equations with coefficients in $A_0$. See the surveys \cite{Pheidas1994}, \cite{PheidasZahidi2000}, \cite{Poonen2008}, and \cite{Koenigsmann14} for a presentation of results and open problems in the context of extensions of Hilbert's tenth problem.

Let $\Rcal$ be the ring  of complex entire functions in one variable $z$ of the form $\sum_{j=1}^n p_j\exp(q_j)$, where $p_j,q_j\in \C[z]$ for each $j$. Holomorphic functions of this type are usually called \emph{exponential polynomials} (of finite order) and there is considerable interest in their value distribution properties going back at least to \cite{Ritt1929}. See  \cite{HITW} and \cite{GSW} for some recent results and an overview of this topic.

The analogue of HTP for $\Rcal$ with coefficients in $\Z$ is decidable by Tarski's Theorem, because $\C\subseteq\Rcal$, and any solution in $\Rcal$ to a system of polynomial equations over the integers can be evaluated at $z=0$ to get a complex solution. Our main result is a negative solution to Hilbert's tenth problem on $\Rcal$ with coefficients in $\Z[z]$. 

\begin{theorem} \label{ThmIntroH10} Let $\Rcal'$ be a subring of $\Rcal$ containing the variable $z$.  The ring $\Z$ is positive existentially interpretable in the ring $\Rcal'$ over the language $\Lcal_z=\{0,1,z,+,\times,=\}$. In particular, the analogue of Hilbert's tenth problem for $\Rcal'$ with coefficients in $\Z[z]$ has a negative answer.
\end{theorem}
In favorable circumstances, we also have a strengthening of the previous result.
\begin{theorem} \label{ThmIntroDef} Let $\Rcal'$ be a subring of $\Rcal$ containing the variable $z$. If the ring of constants $\Rcal'_{\rm cst}=\C\cap \Rcal'$ is positive existentially definable in $\Rcal'$ over $\Lcal_z$, then $\Z$ is positive existentially definable in $\Rcal'$ over $\Lcal_z$. Furthermore, this is the case if $\Rcal'_{\rm cst}=\C$. In particular, $\Z$ is positive existentially definable in $\Rcal$ over the language $\Lcal_z$.
\end{theorem}

Theorems \ref{ThmIntroH10}  and \ref{ThmIntroDef} lie in the context of the study of model-theoretic aspects of rings which may  be constructed by arithmetic operations and composition from the usual functions one encounters in elementary algebra and calculus; this topic dates back at least to Tarski's high school algebra problem, see \cite{Wilkie}. An important case is the ring $\C[z]^E$ obtained from $\Rcal$ by closing under composition of functions; see for instance \cite{Dries1984}, \cite{HensonRubel}, and \cite{HRS}. Elements of $\C[z]^E$ are a more general kind of exponential polynomials, and $\Rcal$ is precisely the subring of $\C[z]^E$ consisting of those holomorphic functions $f\in \C[z]^E$ of finite order (in the sense of growth in complex analysis).

One of the major open problems in the area is the analogue of Hilbert's Tenth Problem for the ring $\Hcal_\C$ of entire holomorphic functions in one variable $z$ with coefficients in $\Z[z]$.  Equivalently, the problem is whether the positive existential theory of the ring $\Hcal_\C$ over the language $\Lcal_z$ is decidable. In more geometric terms, the question is equivalent to asking for an algorithm that takes as input algebraic varieties fibred over the afine line $\pi\colon X\to \A^1$  (all defined over $\Q$) and decides whether there is a complex holomorphic section of $\pi$.

Before discussing how our results fit into the context of Hilbert's tenth problem for $\Hcal_\C$, let us briefly recall some related results. The first order theory of $\Hcal_\C$ over $\Lcal_z$ is undecidable \cite{Robinson1951}. If instead of $\Hcal_\C$ one considers the ring of rigid analytic functions in one variable $z$ over a non-archimedean field $k$, then undecidability of the positive existential theory over $\Lcal_z$ is proved in \cite{LipshitzPheidas1995} when $k$ has characteristic $0$, and in \cite{GarciaFritzPasten2015} when $k$ has positive characteristic. A negative solution to the analogue for Hilbert's tenth problem for rings of complex holomorphic functions in at least two variables is proved in  \cite{PheidasVidaux2018}  over a language including the variables and a predicate for evaluation. Regarding (possibly transcendental) meromorphic functions, much less is known and we refer the reader to \cite{Vidaux2003, Pasten2017, PheidasVidaux2018}. See also  \cite{PheidasZahidi2008} for connections between these problems and questions in number theory. 

The positive existential theory of the ring of complex polynomials $\C[z]$ over $\Lcal_z$ is undecidable \cite{Denef1978}, and there is abundant literature on analogues of Hilbert's tenth problem for \emph{algebraic} function fields and their subrings (which we will not attempt to survey here), but the case of $\Hcal_\C$ offers additional difficulties. 

From a technical point of view, the difficulties in approaching Hilbert's  tenth problem for $\Hcal_\C$ with coefficients in $\Z[z]$ can be attributed to the exponential function $\exp\in \Hcal_\C$, which does not exist in the case of algebraic function fields or non-archimedean rigid entire functions, see \cite{PheidasZahidi2000} for details.  Theorems \ref{ThmIntroH10} and \ref{ThmIntroDef} can be regarded as a step forward in the direction of Hilbert's tenth problem for $\Hcal_\C$ with coefficients in $\Z[z]$, for a subring containing the exponential function. It is worth pointing out that $\Rcal$ is natural subring of $\Hcal_\C$ not just from the point of view of logic, but also from the point of view of value distribution; for instance, in \cite{GSW} it is shown that $\Rcal$ is radically closed in $\Hcal_\C$ by means of Nevanlinna theory.

The proofs of Theorems \ref{ThmIntroH10} and \ref{ThmIntroDef} involve an interaction between Arithmetic, Analysis, Logic, and Functional Transcendence. Let us briefly outline the structure of the argument.

First, the logical side builds on work of \cite{Denef1978} which concerns rings of polynomials and uses Pell equations; see \cite{PheidasZahidi2000} and the references therein for other cases where these ideas are used. However, in our case additional technical difficulties arise, essentially because $f(1)=0$ is not the same as $(z-1)|f$ in $\Rcal$.  

The functional Pell equation that we study is
\begin{equation}\label{Pell}
x^2-(z^2-1)y^2=1
\end{equation}
(in the unknowns $x$ and $y$) over $\Rcal$. The following theorem is our key technical result. 
\begin{theorem} \label{ThmIntroPell} 
Equation \eqref{Pell} has the same solutions over $\Rcal$ and over $\Z[z]$. 
\end{theorem}
Functional Pell equations with polynomial coefficients are those of the form $X^2-DY^2=1$, where $D\in \C[z]$ is a polynomial without multiple zeros. They are some of the oldest studied functional polynomial equations since Abel, see \cite{Zannier2014}, \cite{Zannier2014b}, and \cite{Kollar2019} for some recent developments. 

Equation \eqref{Pell} has an infinite number of \emph{polynomial} solutions, which can be given the structure of an abelian group isomorphic to $\Z\oplus (\Z/2\Z)$. These polynomial solutions are used in \cite{Denef1978} and elsewhere to approach Hilbert's tenth problem over various rings of functions. Roughly speaking, Theorem \ref{ThmIntroPell} allows us to approach Hilbert's tenth problem for $\Rcal$ using the Pell equation method, but there are some serious technical complications (as mentioned before). We explain the details in Section \ref{definabilitysection}. 

The proof of Theorem \ref{ThmIntroPell} is given in Section \ref{transcendencesection} and it has three main ingredients. First, we prove a theorem that gives a complete description of all solutions of Equation \eqref{Pell} in $\Hcal_\C$ in terms of functions in a suitable quadratic extension of $\Hcal_\C$ (Theorem \ref{Analytic}). Secondly, we use an extension of the Borel-Carath\'eodory theorem on an auxiliary Riemann surface which, together with Theorem \ref{Analytic}, allows us to severely restrict the kind of functions in $\Rcal$ that can appear as solutions of Equation \eqref{Pell} ---see Theorem \ref{Growth}. Finally, we use results from functional transcendence (a consequence of the Ax-Schanuel theorem) to show that the solutions of Equation \eqref{Pell} in $\Rcal$ after the restrictions imposed by Theorem \ref{Growth}, have a trivial transcendental part (Proposition \ref{PropAx}).

We finish this introduction by mentioning that there are other natural extensions of the ring language where one may consider a similar problem. For instance, one may  include in the language a predicate symbol for the non-constant functions, i.e. one may want to ask whether a given polynomial equation has or does not have solutions in $\Hcal_\C$ which are not all constant functions. This is an open question ---cf. \cite{PheidasZahidi2008}.\\

%
%
%


\noindent\textbf{Notation and basic facts.}
\begin{itemize}

\item $B$ is the Riemann surface associated to the curve $w^2=z^2-1$. Points in $B$ will be written in coordinates $(z,w)$, and $\pi$ will denote the projection $\pi(z,w)=z$. So, $B$ is a connected Riemann surface and $\pi\colon B\to \C$ is a surjective, proper holomorphic map of degree $2$.

\item $\Mcal_\C$ is the quotient field of $\Hcal_\C$, namely, the field of meromorphic functions on $\C$. 

\item $\Mcal_B$ is the field of complex of meromorphic functions on $B$. It is a quadratic extension of $\Mcal_\C$ by means of the inclusion $\pi^*\colon \Mcal_\C\to\Mcal_B$ defined by pull-back. Indeed, we have $\Mcal_B=\Mcal_\C(w)\simeq \Mcal_\C(\sqrt{z^2-1})$ where $w$ is the holomorphic function on $B$ defined by the $w$-coordinate, as this function satisfies $w^2=z^2-1$.  The extension $\Mcal_B/\Mcal_\C$ is Galois with non-trivial automorphism determined by $w\mapsto -w$. 

\item For $x,y\in\Mcal_\C$, we write $\Nr(x+yw)=(x+yw)(x-yw)$ for the norm of the quadratic extension $\Mcal_B/\Mcal_\C$. 

\item $\Hcal_B$ is the subring of $\Mcal_B$ of holomorphic functions on $B$. Every $h\in \Hcal_B$ can be written in a unique way as $f+gw$, with $f,g\in \Hcal_\C$. Thus, $\Hcal_B=\Hcal_\C[w]$.
\end{itemize}


\noindent\textbf{Acknowledgement.}

This project was initiated at the American Institute of Mathematics meeting ``Definability and decidability problems in number theory'', May 6 to May 10, 2019, San Jose, California. 

D. C., T. P. and X. V. did part of the research while cooperating in the framework of a European Union Erasmus+ project through the University of Crete, Greece and the University of Concepci\'on, Chile.

T. P. did part of the research while visiting University of Napoli Federico II, with an Erasmus exchange, and Institut Henri Poincar\'e during the thematic quarter ``Model Theory, Combinatorics and Valued Fields''. 

The hospitality of all these institutions is greatly appreciated. 

We thank the referee for several insightful comments, as well as technical remarks that led us to a more concise presentation of the paper.


\section{Analytic solutions of Pell's Equation}\label{analyticsection}

Let $t$ be the standard variable on the Riemann surface $\C^\times$ and let $\Hcal_{\C^\times}$ be the ring of complex holomorphic functions on $\C^\times$.
\begin{lemma}\label{cohom} We have $\Hcal_{\C^\times}^\times = \{t^n\cdot \exp(h)\colon n\in \Z \mbox{ and }h\in \Hcal_{\C^\times}\}$.
\end{lemma}
\begin{proof} The exponential exact sequence of sheaves on $\C^\times$ gives the exact sequence in cohomology 
$$
(0)\to \Z\rightarrow \Hcal_{\C^\times}\xrightarrow{\exp} \Hcal_{\C^\times}^\times \xrightarrow{\varphi} H^1(\C^\times,\Z)=\Z,
$$
where $\varphi(f)=(2\pi i)^{-1}\oint_\gamma f'/f$ and $\gamma$ is the circle of radius $1$ around the origin positively oriented. Since $\varphi(t)=1$, we conclude that the class of $t$ in the multiplicative group $\Hcal_{\C^\times}^\times/\exp(\Hcal_{\C^\times})\simeq\Z$ is a generator.
\end{proof}
\begin{lemma}\label{LemmaUnitsB} We have $\Hcal_{B}^\times = \{(z+w)^n\cdot \exp(h): n\in \Z \mbox{ and }h\in \Hcal_{B}\}$.
\end{lemma}
\begin{proof}
The map $\psi\colon B\to \C^\times$ given by $t=\psi(z,w)=z+w$ is an isomorphism of Riemann surfaces. So we can conclude with Lemma \ref{cohom}. 
\end{proof}

For $f,g\in\Hcal_\C$, the norm map $\Nr\colon\Hcal_B^\times \to \Hcal_\C^\times$ satisfies:
$$
\Nr\left((z+w)^n\exp(f+gw)\right)=(z+w)^n\exp(f+gw)\cdot(z-w)^n\exp(f-gw)=\exp(2f).
$$
\begin{lemma} \label{LemmaKerNr} We have $\ker(\Nr\colon\Hcal_B^\times \to \Hcal_\C^\times)=\{\pm (z+w)^n\exp(gw) \colon n\in \Z\mbox{ and }g\in \Hcal_\C\}$.
\end{lemma}
\begin{proof} For $f\in \Hcal_\C$, we have $\exp(2f)=1$ if and only if $f\in \pi i \Z$. We conclude by Lemma \ref{LemmaUnitsB}.
\end{proof}

Writing the elements of $\Hcal_B$ as $x+yw$ for $x,y\in \Hcal_\C$, we have $x+yw\in \ker(\Nr\colon\Hcal_B^\times \to \Hcal_\C^\times)$ if and only if $x^2-(z^2-1)y^2=\Nr(x+yw)=1$. Using Lemma \ref{LemmaKerNr}, this proves the following result. 

\begin{theorem} \label{Analytic} The solutions $(x,y)$ of Equation \eqref{Pell} over $\Hcal_\C$ are of the form
\begin{equation}\label{ana}
x+yw=\pm (z+w)^n \exp(hw)
\end{equation}
where $n$ is a rational integer and $h\in \Hcal_\C$. 
\end{theorem}


\section{Growth}\label{growthsection}

Given an open, connected Riemann surface $S$, a proper degree $n$ holomorphic map $p:S\to \C$, a point $z_0\in \C$, and a holomorphic function $h:S\to \C$, we define the following functions for $r\ge 0$:
$$
M_{S,z_0}(h,r)=\max_{|p(s)-z_0|\le r} |h(s)|, \quad A_{S,z_0}(h,r) = \max_{|p(s)-z_0|\le r} \Re(h(s))
$$
where $\Re$ denotes the real part. In the classical case when $p:S\to \C$ is taken as $\mathrm{Id}_\C:\C\to \C$ and $z_0=0$, we denote these maps simply by $M_\C(h,r)$ and $A_\C(h,r)$.

We will need the following version of the Borel-Carath\'eodory theorem, which follows from Corollaire 7 in \cite{YChen} under the usual convention that a holomorphic map $h:S\to \C$ (with the previous notation) can be seen as a holomorphic multivalued function on $\C$ via the rule $z\mapsto \{h(s) : s\in p^{-1}(z)\}$ (multivalued functions of this type are classically known as algebroid maps.)

\begin{lemma}[Borel-Carath\'eodory for algebroid functions]\label{LemmaBC} Let $S$ be an open, connected Riemann surface having a proper, degree $n$ holomorphic map $p:S\to \C$. Let $h:S\to \C$ be a holomorphic function and let $z_0\in \C$ be such that $\{h(s) : s\in p^{-1}(z_0)\}=\{0\}$.  For all $0<r<R$ we have
$$
M_{S,z_0}(h,r) \le C_n(r,R)\cdot A_{S,z_0}(h,R),
$$
where $C_n(r,R)=2\left((R/r)^{1/n} -1\right)^{-1}$.
\end{lemma}

We prove the following consequence: 

\begin{lemma}\label{LemmaKeyGrowth}
 Let $f,g,h\in \Hcal_\C$ be such that $\exp(hw)=f+gw$ as holomorphic functions on $B$. Then for all  $r\ge 74$ we have
$$
M_\C(h,r) \le \frac{6}{r}\log \max\{M_\C(f,2r),M_\C(g,2r)\} + \frac{12\log r}{r}.
$$
\end{lemma}
\begin{proof} We will be using the previous definitions with $p:S\to \C$ taken as $\mathrm{Id}_\C:\C\to \C$ and $\pi:B\to \C$. First of all, we note that for any $\rho>0$ we have
$$
M_{B,0}(w,\rho)=\max_{|\pi(b)|=\rho}|w(b)| =\max_{|z|=\rho} \sqrt{|z^2-1|}= \sqrt{\rho^2+1}.
$$

For each $b\in B$ we have 
$$
|f(\pi(b))+g(\pi(b))w(b)|=|\exp( h(\pi(b)))w(b)|=\exp \Re\left(h(\pi(b))w(b)\right)
$$
from which we deduce
$$
\begin{aligned}
\exp \left(A_{B,0}(hw,\rho)\right)&=M_{B,0}(f+gw,\rho)\\
&\le M_{B,0}(f,\rho)+M_{B,0}(g,\rho) \cdot M_{B,0}(w,\rho)\\
&= M_\C(f,\rho)+M_\C(g,\rho)\sqrt{\rho^2+1}.
\end{aligned}
$$
Hence,
\begin{equation}\label{EqnBC1}
A_{B,0}(hw,\rho) \le \log \max\{ M_\C(f,\rho), M_\C(g,\rho) \} + \log (2\sqrt{\rho^2+1}).
\end{equation}

Let $r\ge74$. Note that $hw$ satisfies $\{(hw)(b) : \pi(b)=1\}=\{0\}$. Lemma \ref{LemmaBC} applied to $hw:B\to \C$ with $n=2$ and $z_0=1$ gives
$$
M_{B,1}(hw,r+1) \le C_2(r+1,2r-1)\cdot A_{B,1}(hw,2r-1) \le 5\cdot A_{B,1}(hw,2r-1)
$$
where we used the bound $C_2(r+1,2r-1)\le 5$, which is valid because we chose $r\ge 74$. 

Writing $D(z_0,r)=\{z\in \C : |z-z_0|\le r\}$ we observe that $D(0,r)\subseteq D(1,r+1)$ and $D(1,2r-1)\subseteq D(0,2r)$. These inclusions together with the fact that holomorphic functions and their real parts are harmonic, give the bounds 
$$
M_{B,0}(hw,r)\le M_{B,1}(hw,r+1) \le 5\cdot A_{B,1}(hw,2r-1)  \le 5\cdot A_{B,0}(hw,2r)
$$
and from the bound \eqref{EqnBC1} with $\rho=2r$ we conclude
$$
M_{B,0}(hw,r)\le 5\left( \log \max\{ M_\C(f,2r), M_\C(g,2r) \} + \log (2\sqrt{4r^2+1})\right).
$$
Notice that we have
$$
M_{B,0}(hw,r)\ge M_{B,0}(h,r)\cdot \inf_{|\pi(b)|=r}|w(b)|=M_\C(h,r)\cdot \sqrt{r^2-1}  
$$
and 
$$
\frac{5}{\sqrt{r^2-1}}<\frac{6}{r} \quad \mbox{and} \quad 2\sqrt{4r^2+1}<  r^2.
$$
Finally, we get
$$
\begin{aligned}
M_\C(h,r)& \le \frac{5}{\sqrt{r^2-1}} \log \max\{ M_\C(f,2r), M_\C(g,2r) \} + \frac{5\cdot \log (2\sqrt{4r^2+1})}{\sqrt{r^2-1}}\\
&\le \frac{6}{r} \log \max\{ M_\C(f,2r), M_\C(g,2r) \} + \frac{12 \log r}{r}.
\end{aligned}
$$
\end{proof}
\begin{lemma}\label{LemmaCauchyBd} Let $h\in \Hcal_\C$. Let $\alpha\ge 0$ be a real number. If $M_\C(h,r)=O(r^\alpha)$, then $h$ is a polynomial of degree at most $\lfloor \alpha\rfloor$.
\end{lemma}
\begin{proof} This is immediate from Cauchy's estimate for Taylor coefficients.
\end{proof}

We recall that a holomorphic function $f\in \Hcal_\C$ is said to have \emph{finite order} if for some real number $\alpha>0$ one has $M_\C(f,r)=O(\exp(r^\alpha))$. The infimum of all such numbers $\alpha$ is the \emph{order} of $f$.  From the previous two lemmas we deduce:

\begin{theorem} \label{Growth}
Let $f,g,h\in \Hcal_\C$ be such that $\exp(hw)=f+gw$ as holomorphic functions on $B$. If $f$ and $g$ have order at most $\beta\ge1$, then $h\in\C[z]$ is a polynomial of degree at most $\lfloor\beta\rfloor-1$. In particular, if $f$ and $g$ are of finite order, then $h\in \C[z]$. 
\end{theorem}



\section{Functional transcendence and proof of Theorem \ref{ThmIntroPell}}\label{transcendencesection}

In this section, we prove Theorem \ref{ThmIntroPell}. 

Let $K$ be a field containing $\mathbb{C}$. We say that $a_1,\dots,a_n\in K$ are linearly independent over $\mathbb{Q}$ modulo constants if there is no non-trivial relation
$$
\lambda_1a_1+\dots+\lambda_na_n=c
$$
with $c\in\mathbb{C}$ and $\lambda_j\in\mathbb{Q}$ for each $j$. 

The following result is a well known consequence of Ax's theorem \cite{Ax1971} concerning Schanuel's conjecture on power series ---see \cite{Pila2013} or \cite{Pila2015}. 

\begin{theorem}[Ax-Lindemann-Weierstrass theorem]\label{ALW} Let $W$ be an algebraic variety over $\C$, with function field $\C(W)$. If $a_1,\dots,a_n\in\C(W)$ are linearly independent over $\Q$ modulo constants, then $\exp(a_1),\dots,\exp(a_n)$, as functions on $W$, are algebraically independent over $\C(W)$. 
\end{theorem}
\begin{proposition}\label{PropAx} Let $a\in \C[z]$ and $f,g\in \Rcal$. If  $\exp(aw)=f+gw$ as holomorphic functions on $B$, then $a=0$.
\end{proposition}
\begin{proof} 
For the sake of contradiction, suppose that $\exp(aw)=f+gw$ holds with $a\ne 0$. Recalling the definition of $\Rcal$ and rearranging terms, we obtain an identity of the form
\begin{equation}\label{eqalgdep}
\exp(a(z)w) = \sum_{k=1}^N \left(p_k(z) + q_k(z)w\right)\exp(b_k(z))
\end{equation}
for a suitable integer $N$ and certain polynomials $p_k,q_k,b_k\in \C[z]$ (for $k=1,...,N$). Applying Theorem \ref{ALW} with $W=B$, we obtain that $aw$, $b_1$, \dots, $b_N$ are linearly dependent over $\Q$ modulo constants. Since $w$ is quadratic over $\C(z)$, we deduce that $b_1$, \dots, $b_N$ are linearly dependent over $\Q$ modulo constants, so that we have $\sum_{k=1}^{N}\lambda_k b_k=c_0$ for some $\lambda_k\in\Q$ and $c_0\in\C$. Modulo relabelling, we can assume $\lambda_N\ne0$. Let $P_k\in\C[z]$ be such $b_N=c+\sum_{k=1}^{N-1}\alpha_k P_k$, where $c\in\C$ and $\alpha_k\in\Z$ (so $b_k$ is an integer times $P_k$). Replacing $b_1$, \dots, $b_N$ in Eq. \eqref{eqalgdep} by their expression in terms of the $P_k$, we obtain an algebraic relation between $\exp(aw)$, $\exp(P_1)$, \dots, $\exp(P_{N-1})$. By Theorem \ref{ALW}, we deduce that $P_1$, \dots, $P_{N-1}$ are linearly dependent modulo constants. Repeating this process, we get an algebraic relation between $\exp(aw)$ and $\exp(Q)$, hence $Q$ is constant, which is absurd. 
\end{proof}

\begin{proof}[Proof of Theorem \ref{ThmIntroPell}.] Every solution of Equation \eqref{Pell} in $\Z[z]$ is in $\Rcal$. Conversely, let $x,y\in \Rcal$ be a solution of Equation \eqref{Pell}. By Theorem $\ref{Analytic}$, $x$ and $y$ satisfy $x+yw = \pm (z+w)^n \cdot \exp(hw)$ for some $n\in \Z$ and $h\in\Hcal_\C$. Let $f,g\in\Hcal_\C$ be such that $f+gw=(x+yw) \cdot (z+w)^{-n}$. We notice that $f, g$ lie in $\Rcal$ (expanding the product) and $f+gw=\pm \exp(hw)$. Hence, by Theorem \ref{Growth}  we have that $h\in \C[z]$, and by Proposition \ref{PropAx} we have $h=0$. Hence, $x+yw = \pm (z+w)^n$ and we get $x,y\in \Z[z]$ by expanding this expression and using $w^2=z^2-1$.
\end{proof}



\section{Interpretability and definability}\label{definabilitysection}

For $n\in\Z$, we let $x_{n},y_{n}\in \C[z]$ be defined by $x_{n}+y_{n}\sqrt{z^2-1} = (z+\sqrt{z^2-1})^n$. Let us recall from \cite{Denef1978} that
\begin{itemize} 
\item[(D1)] $x_n,y_n\in \Z[z]$ for each $n\in \Z$.
\item[(D2)]  The value of $y_{n}$ at $z=1$ is equal to $n$.
\item[(D3)]  The pairs $(\pm x_n,y_n)$ for $n\in \Z$ are precisely \emph{all} the solutions of Equation \eqref{Pell} in $\C[z]$. 
\end{itemize}

In this section we let $\Rcal'\subseteq \Rcal$ be a subring containing the variable $z$. We remark that $\Z[z]\subseteq \Rcal'\subseteq \Rcal$. As in the introduction, we consider the language $\Lcal_z=\{0,1,z,+,\times,=\}$ and we view $\Rcal'$ as an $\Lcal_z$-structure in the obvious way.

Let us make a slight abuse of notation to simplify formulas: When we write a positive existential $\Lcal_z$-formula to define a set or relation on $\Rcal'$, we can use symbols that stand for functions or relations that are already known to be positive existentially definable on $\Rcal'$ over $\Lcal_z$.

Let us define the set 
$$
\Rcal'_{\Z}=\{f\in \Rcal' : \textrm{ there exists } p \in \Z[z] \textrm{ and } u\in \Rcal' \mbox{ such that } f-p=(z-1)u\}. 
$$
The main reason to work with this slightly technical definition is that it is not clear whether the binary relation $f(1)=g(1)$ is positive existentially definable in $\Rcal'$ over $\Lcal_z$, but in $\Rcal'_\Z$ such difficulties are manageable.

Note that the previous definition of $\Rcal'_\Z$ is not a first-order definition, because we do not know whether $\Z[z]$ is first-order definable in $\Rcal'$. Nevertheless, we will be able to show that $\Rcal'_\Z$ is first-order positive existentially definable in $\Rcal'$. 

\begin{lemma} We have that $\Rcal'_{\Z}$ is a ring and $\Z[z] \subseteq \Rcal'_{\Z} \subseteq \{f\in \Rcal' : f(1)\in \Z\}$.
\end{lemma}
\begin{proof} The inclusions are clear. Also $\Rcal'_{\Z}$ is closed under addition. Regarding multiplication, let $f,g\in \Rcal'_{\Z}$. Take $p,q\in \Z[z]$ and $u,v\in \Rcal'$ such that $f-p=(z-1)u$ and $g-q=(z-1)v$. We have
$$
(z-1)ug=fg-pg=fg-(q+(z-1)v)p= fg-pq-(z-1)pv
$$
hence, $fg-pq=(z-1)(ug+pv)$ where $pq\in \Z[z]$ and $ug+pv\in \Rcal'$.
\end{proof}
\begin{lemma}\label{LemmaKeyLogic} The set $\Rcal'_{\Z}$ and the binary relation $\Val$ defined by
$$
\Val(f,g):\quad  f,g\in \Rcal'_{\Z}\mbox{ and } f(1)=g(1)
$$
are positive existentially definable in $\Rcal'$ over $\Lcal_z$.
\end{lemma}
\begin{proof} Consider the following positive existential $\Lcal_z$-formula on the free variable $T$:
$$
\phi(T):\quad \exists h\exists g, ((h^2-(z^2-1)g^2=1)\wedge (\exists u, T-g=(z-1)u)).
$$
Let $f\in \Rcal'$. We claim that $f\in \Rcal'_{\Z}$ if and only if $\Rcal'$ satisfies $\phi(f)$. 

Assume $f\in \Rcal'_{\Z}$ and let $p\in \Z[z]$, $v\in \Rcal'$ with $f-p=(z-1)v$. Then $n=f(1)=p(1)\in \Z$ and we can take $h=x_n$ and $g=y_n$, which belong to $\Rcal'$ by (D1). With this choice $h^2-(z^2-1)g^2=1$ holds by (D3), and $g(1)=n$ by (D2). Furthermore, $p-g\in \Z[z]$ satisfies $(p-g)(1)=0$, so there is $q\in \Z[z]$ with $p-g=(z-1)q$. Since $q\in \Z[z]\subseteq \Rcal'_\Z$ we have $v+q\in \Rcal'$ and $f-g=(z-1)(v+q)$. Choosing $u=v+q$, we see that $\Rcal'$ satisfies $\phi(f)$.

Conversely, assume that $\Rcal'$ satisfies $\phi(f)$ and choose $h,g,u\in \Rcal'$ with $h^2-(z^2-1)g^2=1$ and $f-g=(z-1)u$. By Theorem \ref{ThmIntroPell} we have $h,g\in \Z[z]$. Since $f-g=(z-1)u$, we get  $f\in \Rcal'_{\Z}$.

Finally, we claim that $\Val(f,g)$ is defined by $f,g\in \Rcal'_{\Z}\wedge \exists h, f-g=(z-1)h$. Indeed, if the formula holds then $f(1)=g(1)$. Conversely, if $f,g\in \Rcal'_{\Z}$ and $f(1)=g(1)$ let us take $p,q\in \Z[z]$ and $u,v\in \Rcal'$ such that $f-p=(z-1)u$ and $g-q=(z-1)v$. Then $p(1)=f(1)=g(1)=q(1)$ and there is $h\in \Rcal'$ (in fact, in $\Z[z]$) with $p-q=(z-1)h$. Hence, $f-g=p-q+(z-1)(u-v)=(z-1)(h+u-v)$, and the formula holds.
\end{proof}
Let us recall that an \emph{interpretation} of a structure $(M_1,\Lcal_1)$ in a structure $(M_2,\Lcal_2)$ is a function $\theta: A\to M_1$ where $A\subseteq M_2^r$ for certain $r\ge 1$, such that 
\begin{itemize}
\item[(Int1)] $\theta$ is surjective onto $M_1$
\item[(Int2)] $A$ is $\Lcal_2$-definable in $M_2$, and
\item[(Int3)] for each symbol $s\in \Lcal_1$ with realization $s^{M_1}\subseteq M_1^k$, the pre-image of $s^{M_1}$ under $k$-th cartesian power of $\theta$ is $\Lcal_2$-definable in $M_2$.
\end{itemize}
In other words, $M_1$ is the quotient of a definable subset of $M_2$, namely $A$, by a definable equivalence relation (induced by the preimages). 

It is a standard fact that if there is an interpretation of $(M_1,\Lcal_1)$ in $(M_2,\Lcal_2)$ and if the first order theory of $(M_1,\Lcal_1)$  is undecidable, then so is the first order theory of $(M_2,\Lcal_2)$.

The interpretation is said to be \emph{positive existential} if the $\Lcal_2$-formulas used in (Int2) and (Int3) are positive existential. In this case, if the positive existential theory of $(M_1,\Lcal_1)$ is undecidable, then so is the positive existential theory of $(M_2,\Lcal_2)$. 

For short, we write ``p.e.'' instead of ``positive existential''.
\begin{proof}[Proof of Theorem \ref{ThmIntroH10}] Let us check that the function $\theta: \Rcal'_{\Z}\to \Z$ given by $\theta(f)=f(1)$ gives a p.e.\ interpretation of  $(\Z; 0,1,+,\times,=)$ in the $\Lcal_z$-structure $\Rcal'$. 

Since $\Z\subseteq \Rcal'_{\Z}$, (Int1) holds. As $\Rcal'_{\Z}$ is p.e.\ $\Lcal_z$-definable in $\Rcal'$ (by Lemma \ref{LemmaKeyLogic}) we get (Int2) with a p.e. $\Lcal_z$-formula. The pre-images of $0^\Z$ and $1^\Z$ are defined by $\Val(f,0)$ and $\Val(f,1)$ respectively. The pre-image of $=^\Z$ is defined by $\Val(f,g)$. The pre-image of $+^\Z$ is defined by $(f,g,h\in \Rcal'_{\Z})\wedge \Val(f+g,h)$, and the case of $\times^\Z$ is similar since $\Rcal'_{\Z}$ is a ring. By Lemma \ref{LemmaKeyLogic}, we obtain (Int3) with p.e.\ $\Lcal_z$-formulas.

Finally, undecidability of the p.e.\ $\Lcal_z$-theory of $\Rcal'$ follows from the fact that the positive existential theory of $(\Z; 0,1,+,\times,=)$  is undecidable.
\end{proof}

\begin{proof}[Proof of Theorem \ref{ThmIntroDef}]  Assume first that $\Rcal'_{\rm cst}$ is p.e.\ $\Lcal_z$-definable in $\Rcal'$. We note that $\Z=\Rcal'_{\rm cst}\cap \Rcal'_{\Z}$, so $\Z$ is p.e.\ $\Lcal_z$-definable in $\Rcal'$ by Lemma \ref{LemmaKeyLogic}.

In the special case when $\C\subseteq \Rcal'$ we note that for $v\in \Rcal'$, we have $v\in \C$ if and only if $\Rcal'$ satisfies the formula $\exists f, (v^2=f^5+1)$. This is because the curve $y^2=x^5+1$ has geometric genus $2$, so it admits no non-constant meromorphic parametrizations by Picard's theorem.
\end{proof}



%
%
%


\end{document}